\documentclass[draft]{publmathdeb}
\usepackage{amsmath,amsfonts,amssymb}
\usepackage{latexsym}

\def \X{{\mathfrak X}}
\newcommand{\be}{\begin{equation}}
\newcommand{\ee}{\end{equation}}
\newcommand{\bea}{\begin{eqnarray}}
\newcommand{\eea}{\end{eqnarray}}
\newcommand{\ba}{\begin{array}}
\newcommand{\ea}{\end{array}}
\newcommand{\bi}{\begin{itemize}}
\newcommand{\ei}{\end{itemize}}
\newcommand{\bc}{\begin{center}}
\newcommand{\ec}{\end{center}}
\newcommand{\bfr}{\begin{flushright}}
\newcommand{\efr}{\end{flushright}}
\newcommand{\f}{\frac}

\newcommand{\la}{\langle}
\newcommand{\ra}{\rangle}

\newcommand{\p}{\partial}

\newcommand{\ds}{\displaystyle}

\newcommand{\q}{\quad}

\newcommand{\lam}{\lambda}

\newcommand{\til}{\widetilde}
\newcommand{\bb}{\mathbb}
\newcommand{\cal}{\mathcal}

\newtheorem{theorem}{Theorem}[]
\newtheorem{lemma} [theorem]{Lemma}
\newtheorem{definition} [theorem]{Definition}

\title[A compactness theorem in Finsler geometry]{A compactness theorem in Finsler geometry}

\author{Mihai Anastasiei}
\address{Faculty of Mathematics, Alexandru Ioan Cuza University of Ia\c si  and
Mathematical Institute ``O. Mayer'', Romanian Academy, Ia\c si, Romania}
\email{anastas@uaic.ro}

\author{Ioan Radu Peter}
\address{Department  of Mathematics, Technical University of Cluj-Napoca, Memorandumului 28, Cod RO-400114, Cluj-Napoca, Romania}
\email{Ioan.Radu.Peter@math.utcluj.ro}


\submitted{}

\begin{document}
\dedicatory{Dedicated to Professor Dr. Lajos Tam\'{a}ssy on the occasion
     of his $90$th birthday}
\begin{abstract}
Let $(M.F)$ be a complete Finsler manifold and $P$ be a minimal and compact submanifold of $M$. $\mathbf{Ric}_k(x), x\in M$ is a differential invariant that interpolates between the flag curvature and the Ricci curvature. We prove that if on any geodesic $\gamma(t)$ emanating orthogonally from $P$ we have $\int_{0}^{\infty}\mathbf{Ric}_{k}(t)>0$, then $M$ is compact.
\end{abstract}

\maketitle

\section*{Introduction}

The classical Gauss-Bonnet Theorem opened a series of results that  are extracting topological properties of a differentiable manifold from the various properties of certain differential geometric invariants of that manifold. The basic topics in this framework consist of the Hopf-Rinow Theorem, the theory of Jacobi fields and the relationship between geodesics and curvature, the Theorems of Hadamard, Myers, Synge, the Rauch Comparison Theorem, the Morse Index Theorem and others. In the Finslerian setting the most recent account of results of this type is due to D. Bao, S.S. Chern and Z. Shen in \cite{BCS}, Ch. 6-9. For a weakened version of the Myers theorem we refer to \cite{AN1}.

The main differential geometric invariants involved in these results are the flag curvature and the Ricci curvature. Among the many others there exits one denoted by $\mathbf{Ric}_k$ and called $k-$ Ricci curvature that interpolates between the flag curvature and the Ricci curvature.
In this paper we consider an $n-$ dimensional, complete Finsler manifold $(M,F)$ , a minimal, compact submanifold $P$ of it and we prove that if on any geodesic $\gamma(t)$ on $M$ emanating orthogonally from $P$ we have $\int_{0}^{\infty}\mathbf{Ric}_{k}(t)>0$, then $M$ is compact. For the Riemannian case there are many similar results (see \cite{BT} and the references therein). By our knowledge, in the Finslerian case, our result is the first of this type but the techniques we use here can be adapted to find many others.  The differential invariant $\mathbf{Ric}_k$ was deeply studied by Z. Shen. In \cite{shen2}, he proves various results  concerning  the vanishing of homotopy groups
under the assumption that the $k$-Ricci curvature satisfies $\mathbf{Ric}_k\geq k$.

We outline the proof of our result. Considering the submanifold $P$ the notion of conjugate points is replaced with that of focal points. The Morse index form written on a geodesic emanating from $P$ takes a special form that involves the second fundamental form of $P$ (see \cite{Pe1}). The conditions $M$ complete but non-compact, and $P$ compact imply that any geodesic emanating orthogonally from $P$ is free of focal points. But choosing a convenient orthogonal frame along the geodesic emanating from $P$, we reduce the Jacobi equation to a scalar differential equation of order two that  by our hypothesis on $\mathbf{Ric}_k$  admits on $[0,\infty )$ a  solution with at least one  zero. In combination with a form of the Index Lemma from \cite{Pe1} one yields that the said geodesic has focal points. The contradiction shows that $M$ has to be compact.

In the Sections 1-3 we prepare all is needed for the detailed proof given in the Section 4.

\section{Preliminaries}
Let $M$ be a real manifold of dimension $n$ and $(TM, \pi, M)$ its tangent bundle. The vertical bundle of the manifold $M$ is the vector subbundle of the double tangent bundle $TTM$ denoted by $({\mathcal V},\tilde \pi, TM)$ and defined  by ${\mathcal V}=\mbox{Ker}\, d\pi\subset T(TM)$, where $d\pi$ is the linear tangent map to $\pi$. Let $(x^i)$ denote the local coordinates on an open subset $U$ of $M$, and $(x^i,y^i)$ the induced
coordinates on $\pi^{-1}(U)\subset TM$. The radial vector field $\iota$ is the vertical vector field locally given by $\iota(x,y)=y^i\ds\f{\p}{\p y^i}$.

A Finsler metric on $M$ is a function $F:TM\to\bb{R}_+$ satisfying the following properties:
\begin{enumerate}
\item $F^2$ is smooth on $\til{M}$ where $\til{M}=TM\setminus {0}$
\item $F(u)>0$ for all $u\in\til{M}$
\item $F(\lam u)=|\lam|F(u)$ for all $u\in TM$, $\lam\in\bb{R}$
\item For any $p\in M$ the indicatrix $I_p=\{u\in T_pM|\ F(u)<1\}$ is strongly convex.
\end{enumerate}

A manifold $M$ endowed with a Finsler metric $F$ is called a Finsler manifold $(M,F)$.

>From the condition (4) it follows that the quantities $g_{ij}(x,y)=\ds\f{1}{2}\ds\f{\p^2 F^2(x,y)}{\p y^i \p y^j}$  form the entries of
a positive definite matrix so a Riemannian metric $\la\cdot,\cdot\ra$ can be introduced in the vertical bundle $({\mathcal V},\til{\pi},TM)$.

On a Finsler manifold there is not, in general, a linear metrical connection. However, there are several metrical connections and among them the analogue of the Levi-Civita connection in the vertical bundle of $(M,F)$.

We will use the Cartan connection which is a good vertical connection on ${\mathcal V}$, i.e. an $\bb{R}$-linear map
$$\nabla^v:{\X}(\til{M})\times{\X}({\mathcal V})\to {\X}({\mathcal V})$$ having the usual properties of a covariant
derivative , is metrical with respect to $\langle \cdot, \cdot \rangle$, and 'good' in the sense that the bundle map
$\Lambda:T\til{M}\to{\mathcal V}$ defined by $\Lambda(Z)=\nabla_Z^v \iota$ restricted to ${\mathcal V}$ is a bundle isomorphism. The latter property induces the horizontal subspaces $H_u=\mbox{Ker}\,\Lambda$ for all $u\in \til{M}$ which are direct summands of the vertical subspaces
$V_u=\mbox{Ker}\,(d\pi)_u$. They define a vector bundle called the  horizontal bundle ${\mathcal H}$ such that $$T\til{M}={\mathcal H}\oplus{\mathcal V}.$$

For a tangent vector field $X$ on $M$ we have its vertical lift $X^V$ and its horizontal lift $X^H$ to $\til{M}$.

Let be $\delta_i=\left(\dfrac{\partial}{\partial x^i}\right)^H$. These local vector fields provide a local basis for the distribution ${\mathcal H}$.  Define $\Theta:{\mathcal V}\to{\mathcal H}$ as the vector bundle morphism locally given  by $\Theta \left(\dfrac{\partial}{\partial y^i}\right)= \delta_i$. It is in fact the inverse of the mapping $\Lambda$ and is clearly an isomorphism of vector bundles. It is called the horizontal map associated to the horizontal bundle ${\mathcal H}$.

Using $\Theta$, first we get the radial horizontal vector field $\chi=\Theta\circ\iota$. For a curve $\sigma(t)$ on $M$ let be $\dot{\sigma}$ its tangent vector field. Its horizontal lift ${\dot\sigma}^H$ is just $\chi$ in the point $\dot{\sigma}$ of $TM$. Locally, ${\dot{\sigma}}^H=\dfrac{d\sigma ^i}{dt}\delta_i$.

Secondly we can extend the covariant derivation $\nabla^v$ of the vertical bundle to the whole tangent bundle of $\til{M}$. Denoting it with $\nabla$, for horizontal vector fields $H$ we set $$\nabla_Z H=\Theta(\nabla_Z^v(\Theta^{-1}(H))),\ \forall \ Z\in {\X}(\til{M}).$$

An arbitrary vector field $Y\in {\X}(\til{M})$ is decomposed into vertical and horizontal parts: $$\nabla_Z Y=\nabla_Z Y^V+\nabla_Z Y^H.$$

Thus $\nabla:{\X}(T\til{M})\times {\X}(T\til{M})\to {\X}(T\til{M})$ is a linear connection on $\til{M}$ induced by a good vertical connection. Its torsion $\theta $ and curvature $\Omega$ are defined as usual:
\begin{eqnarray*}
\theta (X,Y)= \nabla _{X}Y-\nabla_{Y}X -[X,Y],
&\Omega(X,Y)Z=\nabla_X\nabla_YZ-\nabla_Y\nabla_XZ-\nabla_{[X,Y]}Z
\end{eqnarray*}
and the torsion has the property that for horizontal vectors, $\theta(X,Y)$ is a vertical vector \cite{Ab:Pa1}.

The Riemannian metric $\la\cdot ,  \cdot\ra$ on $\mathcal{V }$ can be moved to a Riemannian metric on the vector bundle  $\mathcal{H}$ and  these two Riemannian metrics provide a Riemannian metric on $\tilde {M}$ (a Sasaki type metric) just by stating that $\mathcal{H}$ is orthogonal to $\mathcal{V}$. All these metrics will be denoted with the same symbol whose meaning will be clear from context.

The metrical property of the  connection $\nabla$ holds good:

$$ X\langle Y, Z\rangle=\langle \nabla_{X}Y, Z\rangle+\langle Y, \nabla_{X}Z\rangle.$$

The sectional curvature of $\nabla$ along a curve $\sigma$ is given as follows:
$$K_{\dot\sigma}(U^H,U^H) = \langle \Omega(\dot\sigma^H, U^H)\dot\sigma^H,U^H\rangle$$ for any $U\in \X(M)$. This is called the
horizontal flag curvature in \cite{Ab:Pa1}.

\section{The Morse Index form}
We recall some facts about the variation of energy and Morse Index form.\cite{Pe1}.

\begin{definition}\cite{Ab:Pa1}. A \textit{regular curve} $\sigma:[a,b]\to M$ is a $C^1$-curve such that
$$\forall \ t\in [a,b]\q T(t) \equiv \dot{\sigma}(t)=d\sigma_t\left(\f{d}{dt}\right)\ne 0.$$

The length, with respect to the Finsler metric $F:TM\to\bb{R}^+$, of the regular curve is given by
$$L(\sigma)=\int_a^b F(\dot{\sigma}(t))dt$$ and the energy is given by $$E(\sigma)=\int_a^b F^2(\dot{\sigma}(t))dt.$$
\end{definition}

The critical points of the (length) energy functional are the geodesics $\sigma$ in the Finsler manifold $M$, whenever they are parameterized by arc-length i.e. $F({\dot{\sigma}}) =1$. A geodesic parameterized by the arc-length will be called normal. One proves that the geodesics are characterized also by

\begin{theorem}\cite{Ab:Pa1}
A regular curve $\sigma_0$ is \textit{geodesic} for $F$ if and only if $$\nabla_{T^H}T^H\equiv 0$$
where $T^H(u)= {\dot{\sigma}}^H=\chi_u(\dot{\sigma }(t))\in {\cal H}_u$ for all $u\in\til{M}_{\sigma (t)}$.
\end{theorem}

The second variation formula is providing the Jacobi fields and suggests the consideration of the index form. It is derived by using a two parameters geodesic variation. For details we refer to  \cite{Ab:Pa1}, \cite{Ma1}, \cite{Pe1}.

Let $\sigma:[a,b]\to M$ be a normal geodesic in a Finsler manifold $M$. We will denote by
${\X}[a,b]$ the space of piecewise smooth vector fields $X$ along $\sigma$ such that
$$\la X^H,T^H\ra_T\equiv 0.$$

\begin{definition}\cite{Ab:Pa1}
The \textit{Morse index form} $I=I_a^b:{\X}[a,b]\times{\X}[a,b]\to \bb{R}$ of the normal geodesic $\sigma :[a,b]\to M$
is the symmetric bilinear form $$I(X,Y)=\int_a^b [\la \nabla_{T^H}X^H,\nabla_{T^H}Y^H\ra_T-\la\Omega(T^H,X^H)T^H,Y^H\ra_T]dt$$ for all $X,Y\in
{\X}[a,b]$.
\end{definition}
After some computations one gets another formula for the Morse index form \cite{Ab:Pa1}:
$$I(X,Y)=\la\nabla_{T^H}X^H,Y^H\ra_T\Big|_a^b-\int_a^b \la\nabla_{T^H}\nabla_{T^H}X^H+\Omega(T^H,X^H)T^H,Y^H\ra_T dt.$$

\begin{definition}\cite{Ab:Pa1}
A \textit{Jacobi field} along a geodesic $\sigma:[a,b]\to M$ is a vector field $J$ which satisfies the Jacobi
equation $$\nabla _{T^H}\nabla _{T^H}J^H+\Omega(T^H,J^H)T^H\equiv 0$$ where $J^H(t)=\chi_{\dot{\sigma }(t)}(J(t))$.
\end{definition}
$\dot{\sigma}$ and $t\dot{\sigma}$ are Jacobi fields; the first one never vanishes, the second one vanishes only at $t=0$.

Two points $\sigma(t_0)$ and $\sigma (t_1)$, $t_0,t_1\in [a,b]$ are said to be conjugate along $\sigma $ if there exists a nonzero Jacobi field $J$ along $\sigma$ with $J(t_0)=0$ and $J(t_1)=0$.

\section{Minimal submanifolds. Focal points}

Let $P$ be a submanifold of $M$ of dimension $r<n$. We consider the set $$A=\{(x,v)|x\in P, v\in T_xM \}=\{\til x\in
\til{M}|\pi(\til x)\in P\}.$$

Let $H_{\til x}T_xM $ and $H_{\til x}T_xP$ be the horizontal lifts of $T_xM$ and $T_xP$ to $\til x$ and
\begin{displaymath}
H_PTM=\bigcup_{\til x\in A}H_{\til x}T_xM
\end{displaymath}
and
\begin{displaymath}
H_PTP=\bigcup_{\til x\in A}H_{\til x}T_xP.
\end{displaymath}
 For horizontal vector fields
$X,Y\in H_PTP$ let $X^*,Y^*$ be some prolongations of them to $H_PTM$. The restriction of $\nabla_{X^*}Y^*$ to $\til{P}= TP\setminus 0$ does not depend of the choice of
the prolongations.

Let $P_{\til{x}}^\perp$ be the $\la\cdot,\cdot \ra_{\til{x}}$ orthogonal complement of $H_{\til{x}}TP$ in $H_{\til{x}}TM$. By the orthogonal decomposition $$H_{\til x}T_xM=H_{\til x}T_xP\oplus P_{\til{x}}^\perp, {\til x} = (x,v)\in A $$ we obtain that
$$\nabla_{X^*}{Y^*}=\nabla^*_{X}{Y}+\mathbb{I}_v(X,Y).$$ We will call $\mathbb{I}_v(X,Y)$
the second fundamental form at $X$ and $Y$ in the direction of $v$. Note that for $\tilde x=(x,v)$ with $v\in T_xM\setminus T_xP$ we have
\begin{eqnarray}\label{second}
\la\nabla_{X^*}{Y^*},v^H\ra_v=\mathbb{I}_v(X,Y).
\end{eqnarray}

\begin{definition}
Let $P\subset M$ be an $r$-dimensional submanifold of a Finsler manifold $(M,F)$. The submanifold $P$ is called minimal if  for every tangent vector $v$ to $M$ and for any horizontal orthogonal vectors $V_i^H, i=\overline{1,r}$ (i.e. $\la V_i^H, V_j^H\ra_v=0$ for $i\neq j$) we have
$\sum_{i=1}^{r}\mathbb{I}_{v}(V_i^H,V_i^H)=0$.
\end{definition}

The condition of minimality is equivalent with the vanishing of the trace of the linear operator $A_{v^H}$, where $A_{v^H}$
is the linear operator defined by
$$\la A_{v^H}X^H,Y^H\ra_v=\la\mathbb{I}_T(X^H,Y^H),v^H\ra_v.$$
For details we refer to \cite{Dr1}, \cite{shen1}.

Now let $\sigma:[a,b]\to M$ be a normal geodesic in $M$ with $\sigma
(a)\in P$ and $\dot{\sigma}^H(a)$ in the normal bundle of $P$ (i.e. $\dot{\sigma}^H(a)\perp(H_{\dot{\sigma}(a)}T_{\sigma (a)}P)$).

Let $\til{\X}^P={\X}^P[a,b]$ be the vector space of all piecewise smooth vector fields $X$ along $\sigma $ such that $X^H(a)\in
T_{\dot{\sigma}(a)}\til{P}$ and let ${\X}^P$ be the subspace of $\til{\X}^P$ consisting of these $X$ such that $X^H$ is orthogonal
to $\dot{\sigma}^H$ along the curve $\sigma $.

We have that
\begin{eqnarray}\label{second1}
\la\nabla_{T^H}X^H,Y^H\ra_T&=&\la\nabla_{X^H}T^H+[T^H,X^H]+\theta(T^H,X^h),Y^H\ra_T\\
\nonumber &=&\la\nabla_{X^H}T^H,Y^H\ra_T,
\end{eqnarray}
because $[T^H,X^H]$ and $\theta(T^H,X^h)$ are vertical vectors (\cite{Ab:Pa1}).

And for $Y^H$ orthogonal to $T^H$ we have that
\begin{eqnarray}\label{second2}
0=X^H\la T^H,Y^H\ra_T=\la\nabla_{X^H}T^H,Y^H\ra_T+\la T^H,\nabla_{X^H}Y^H\ra_T.
\end{eqnarray}
By considering the vector fields $X^H,Y^H$ such that $X^H(a),Y^H(a)\in T_{\dot{\sigma }(a)}\til{P}$ and taking account of formulas
(\ref{second}), (\ref{second1}), (\ref{second2}) the Morse index form $I^P: \X^P\times \X^P\rightarrow \mathbb{R}$, becomes
$$I^P(X,Y)=\la\nabla _{T^H}X^H,Y^H\ra_T\Big|^b+\la \mathbb{I}_{T}(X^H,Y^H),T^H\ra_T\Big|_a$$
$$-\int_a^b \la\nabla _{T^H}\nabla _{T^H}X^H+\Omega(T^H,X^H)T^H,Y^H\ra_T dt.$$

From \cite{Pe1} we know that $I^P$ is symmetric.

\begin{definition}\cite{Pe1}
Let $P\subset M$ be an $r$-dimensional submanifold of a Finsler manifold $(M,F)$. A $P$-Jacobi field $J$ is a Jacobi field which satisfies in addition
$$J(a)\in T_{\sigma (a)}P$$ and
$$\la\nabla _{T^H}J^H+A_{T^H}J^H,Y^H\ra_T\Big|_a=0$$ for all $Y\in (T_{\sigma (a)}P)^H$.
\end{definition}
The last condition means in fact that
$$\nabla_{T^H}J^H+A_{T^H}J^H\in((T_{\sigma(a)}P)^H)^\perp.$$

The dimension of the vector space of all $P$-Jacobi fields along $\sigma $ is equal to $r$ and the dimension of the vector space of the Jacobi
fields satisfying $$\la J^H,T^H\ra=0$$ is equal to $r-1$.

If $P$ is a point, then a $P$-Jacobi field is a Jacobi field $J$ along $\sigma $ such that $J(a)=0$.

 A point $\sigma (t_0)$, $t_0\in [a,b]$ is said to be a $P$-focal point along $\sigma $ if there exists a non-null $P$-Jacobi field $J$ along $\sigma$ with $J(t_0)=0$.

We shall use the following  Lemma from \cite{Pe1}.

\begin{lemma}\label{mon}
Let $(M,F)$ be a Finsler manifold and $\sigma:[a,b]\to M$ be a geodesic, and $P\subset M$ be a submanifold of $M$. Suppose that
there are no $P$-focal points along $\sigma$. Let $X,J\in\til{\X}^P$ be vector fields orthogonal to $\sigma$ with $J$ a
$P$-Jacobi field such that $X(b)=J(b)$. Then $$I^P(X,X)\ge I^P(J,J)$$ with equality if and only if  $X=J$.
\end{lemma}

\section{Our result}

First, we introduce  the $k$-Ricci curvature, following \cite{shen2}. For a $(k+1)$-dimensional subspace
$\mathcal{V}\in T_xM$ the Ricci curvature $\mathbf{Ric}_y\mathcal{V}$ on $\mathcal V$ is the trace of the
Riemann curvature restricted to $\mathcal{V}$, with flagpole $y$, and is given by:
$$\mathbf{Ric}_{y}(\mathcal{V})=\sum_{i=1}^{k} \la R_y(b_i), b_i\ra_y=\sum_{i=1}^{k}\langle \Omega (y, b_i)y,
b_i\rangle_y,$$  where $R_y(b_i) \equiv \Omega (y,b_i)y$ and $y, (b_i)_{i=\overline{1,\dots,k}}$ is an arbitrary orthonormal basis for $(\mathcal{V},\la,\ra_y)$, with $b_{k+1}=y$.
 $\mathbf{Ric}_{y}(\mathcal{V})$ is well-defined and is positively
homogeneous of degree two on $\mathcal{V}$, $$\mathbf{Ric}_{\lambda y}(\mathcal{V})=\lambda^2\mathbf{Ric}_{y}(\mathcal{V}), \ \
\mbox{for}\ \  \lambda >0, y\in \mathcal{V}.$$

It is clear from the definition that $\mathbf{Ric}_y(T_xM)$ is nothing but the Ricci curvature $\mathbf{Ric}(y)$ for $  y\in T_xM$.

If $\mathcal{V}=P\subset T_xM$ is a tangent plane, the flag curvature is given by
$$K(P,y)=\frac{\la R_y(u),u\ra_y}{\la y,y\ra_y \la u,u\ra_y-\la u,y\ra_y},$$ where $u\in P\setminus \{0\}$, $\mbox{span} (y,u)=P$. This is independent of the choice of $u\in P\setminus\{0\}$, and for $u$ being $g_y$ orthogonal to $y$ and of $g_y-$ norm $1$ it becomes
$$\mathbf{K}(P,y)=\frac{\mathbf{Ric}_y{P}}{F^2(y)},\ \ \ y\in P.$$

Consider the following function on $M$: $$\mathbf{Ric}_{k}(x):= \inf_{\dim (\mathcal{V})=k+1}\inf_{y\in \mathcal{V}}\dfrac{\mathbf{Ric}_y(\mathcal {V})}{F^2(y)},$$ the infimum being considered over all $(k+1)$-dimensional subspaces $\mathcal{V}\subset T_xM$ and $y\in\mathcal{V}\setminus\{0\}$. From the above definitions it can be seen that $$\mathbf{Ric}_1\leq \dots \leq\frac{\mathbf{Ric}_k}{k}\leq \dots \leq \frac{\mathbf{Ric}_{n-1}}{n-1},$$
and $$\mathbf{Ric}_1=\inf_{(P,y)}\mathbf{K}(P,y)\ \  \mbox{and} \ \ \mathbf{Ric}_{(n-1)}=\inf_{F(y)=1}\mathbf{Ric}(y).$$

We will say that the Finsler manifold $(M,F)$ has positive $k$-Ricci curvature if and only if $\mathbf{Ric}_k>0$.

Secondly, we recall a result from the theory of differential equations which will be essential in the proof of our  result.

\begin{theorem} (\cite{tip})
Consider the differential equation $$f''(t)+H(t)f(t)=0~,~t\in [0,\infty)$$ with $H(t)$ continuous.  If
$$\int_{0}^{\infty}H(t)dt>0$$ there exists a solution $f$ satisfying the conditions $f(0)=1$, $f'(0)=0$ and there exists $t_0>0$ for which $f(t_0)=0$.
\end{theorem}

Here $ \int_{0}^{\infty}$ means $\lim_{l\rightarrow {\infty}}\int_0^l$. The conditions satisfied by the solution $f$ are similar to those meet in the definition of  focal points. A differential equation $f''(t)+H(t)f(t)=0$ admitting such a solution $f$ will be called  {\it focal}. There are several other sufficient conditions for a differential equation $f''(t)+H(t)f(t)=0$ be focal, \cite{Ga1}, \cite{Ga2}.

Now we state and prove our  result.

\begin{theorem}
Let $(M,F)$ be a $n$-dimensional  complete Finsler manifold and $P$ be a $r$-dimensional compact and minimal submanifold of $M$. If the $k$-Ricci curvature satisfies the condition
$$\int_{0}^{\infty}\mathbf{Ric}_{k}(t)>0$$
along any geodesic $\gamma:[0,\infty)\rightarrow M, t\rightarrow \gamma(t)$ emanating orthogonally from $P$, then $M$ is compact.
\end{theorem}

\begin{proof}
Suppose, by contrary that $M$ is not compact. Then there exists a normal geodesic $\gamma(t)$ emanating from $P$ and orthogonal to $P$ free of focal points, i.e. there exists a sequence of ${p_i}$ such that the distance $d(p_i,P)$ tends to infinity, since it is supposed that $M$ is non-compact.
By the completeness of $M$ and the compactness of $P$ there exists for each $p_i$ a normal geodesic $\gamma_i$ which realizes the minimum distance $d(p_i,P)$. Denote by $x_i$ the point in $P$ which is joined with $p_i$ by $\gamma_i$, $\gamma_i(0)=x_i\in P$, $\gamma(1)=p_i$. It is known that the geodesic $\gamma_i$ intersects $P$ ortogonally with respect to the inner product $\la~,~\ra_{\gamma'_i(0)}$, that is $T_i=\gamma'_i(0)$ is orthogonal to $P$ with respect to $\la~,~\ra_{\gamma'_i(0)}$. By the compactness of $P$ there exists an accumulation point $x\in P$ of the sequence ${x_i}$ and also $T_i\rightarrow T$ with $T \perp P$ with respect to $\la~,~\ra_T$ and $F(T)=1$. It follows that the length of the geodesic $\gamma(t)$ with initial data $(x,T)$ is equal to $d(x,\gamma(t))$, so $\gamma(t)$ is $P$-focal point free.

On the other hand from the conditions in the theorem we will show that $\gamma(t)$ has $P$-focal points. This contradiction shows that $M$ has to be compact.

The index form along the geodesic $\gamma$ with variations vector field $V$ is
\begin{eqnarray}
I^P(V, V)&=&\int_0^l [\la \nabla_{T^H}V^H,\nabla_{T^H}V^H
\ra_T-\la\Omega(T^H,V^H)V^H,T^H\ra_T]dt  \\
&=&\la \nabla_{T^H}V^H, V^H\ra_T\Big|^l    + \la \mathbb{I}_{T}(V^H,V^H),T^H\ra_T\Big|_0 \nonumber \\
&&-\int_0^l \la\nabla _{T^H}\nabla _{T^H}V^H+
\Omega(T^H,V^H)T^H,V^H\ra_T dt. \nonumber
\end{eqnarray}

We are going to use the parallel transport with reference vector $T$. We construct a moving frame ${V_i(t)}$, $i=\overline{1,r}$ along $\gamma$ such that
\begin{itemize}
\item $V_i(0)$ is an orthogonal basis in $T_{\gamma(0)}P$ and $\la V_i^H(0), T^H(0)\ra_{T(0)}=0$
\item $V_i(t)$ are parallel  along $\gamma$, i.e. $\nabla_{T^H}V_i^H=0$.
\end{itemize}
It follows that the $V_i^H(t)$  are  orthogonal to each other and to $T^H(t)$ along $\gamma$
with respect to the inner product $\la~,~\ra_T(t)$.

We have, for $i=\overline{1,r}$
\begin{eqnarray}
I^P(V_i, V_i)&=&\la \nabla_{T^H}V_i^H, V_i^H\ra_T\Big|^l+\la \mathbb{I}_{T}(V_i^H,V_i^H),T^H\ra_T\Big|_0 \\
&&-\int_0^l \la\nabla _{T^H}\nabla _{T^H}V_i^H+\Omega(T^H,V_i^H)T^H,V_i^H\ra_T dt. \nonumber
\end{eqnarray}

We summing up from $i=1$ to $r$. Since $P$ is minimal we have $$\sum_{i=1}^{k}\la\mathbb{I}_{T}(V_i^H,V_i^H),T^H\ra_T\Big|_0=0 $$ and one yields

\begin{eqnarray*}
\sum_{i=1}^{r}I(V_i,V_i)&=&\sum_{i=1}^{r}\la \nabla_{T^H}V_i^H, V_i^H\ra_T\Big|^l\\
&&-\sum_{i=1}^{r} \int_0^l \la\nabla _{T^H}\nabla_{T^H}V_i^H+
\Omega(T^H,V_i^H)T^H,V_i^H\ra_T dt~.
\end{eqnarray*}

Let us take $X_i(t)=f(t)V_i(t)$ with $f:[0,\infty )$ satisfying $f(0)=1, f'(0)=0$.  Then
$$X_i(0)=V_i(0),X'_i(t)=f'(t)V_i(t), X"(t) = f"(t)V_i.$$

It follows that
\begin{eqnarray}
\sum_{i=1}^{r}I(X_i,X_i)= r f(t)f'(t)\Big|^{l}  -r \int_0^{l}(f''(t)+f(t)\frac{1}{r}\mathbf{Ric}_T(\mathcal {V})f(t)dt ,
\end{eqnarray}
where $\mathcal{V}$ is the linear space spanned by $T, V_i, i=1,2...r$.

In our hypothesis on $\mathbf {Ric}_k$, setting $ rH = \mathbf{Ric}_T(\mathcal {V})$ it comes out that the equation $f"(t) +f(t)H(t) =0$ is focal. By the Theorem 10, there exists $t_0>0$ such that  $f(t_0)=0$. We take $l=t_0$. In the r.h.s. of (6) the first term vanishes because of $f(t_0)=0$ and the second is null since $f$ is a solution of the focal equation $f"(t) +f(t)H(t) =0$. Thus (6) reduces to $\sum_{i=1}^{r}I(X_i,X_i)=0$. It follows that  there exists $X_i$ with $I(X_i,X_i)\leq0$.

Then, the Lemma \ref{mon}  implies that there exists $P$-focal points on the geodesic $\gamma$, which contradicts the assumption that $M$ is not compact. It follows that $M$ has to be compact.
\end{proof}

{\bf Remark.}
For Berwald manifolds our result follows directly from \cite{Ga1}. Moreover, for Berwald manifolds some results from \cite{Ga2} holds.  Indeed, in the case of Berwald manifolds the connection of the Berwald metric  lives on the tangent level (the referrence vector is irrelevant).  Szab\'o's structure theorems (see \cite{Sz}) implies that there exists a non-unique Riemannian metric $g$ on $M$ such that the Berwald connection is the connection of the Riemannian metric. Taking into account that the flag curvature of the Berwald metric is equal to the sectional curvature of the Riemannian metric $g$ and the second fundamental form of a submanifold with respect to the Berwald metric will be the analogue counterpart of the Riemannian metric, the results from \cite{Ga1},\cite{Ga2} apply.

{\bf{Acknowledgements.}} The first author was partially supported by a grant of the Romanian National Authority for Scientific Research, CNSS-UEFISCDI, project number PN-II-ID-PCE-2011-3-0256.

The work of the second author has been co-funded by the Sectoral Operational Programme Human Resources Development 2007-2013 of the Romanian Ministry of Labour, Family and Social Protection through the Financial Agreement POSDRU/89/1.5/S/62557and the RO-Hu bilateral cooperation program.


\begin{thebibliography}{99}

\bibitem{Ab:Pa1}
M.~Abate and G.~Patrizio.
\newblock {\em Finsler Metrics - A Global Approach}, volume 1591 of {\em
  Lecture Notes in Mathematics}.
\newblock Springer Verlag, Berlin, Heidelberg, 1994.

\bibitem{AN1}Mihai Anastasiei
\newblock A Generalization of Myers Theorem.
\newblock {\em An. \c Stiin\c t. Univ. "A. I. Cuza" Iasi, Mat. (NS0 53(2007), suppl. 1, 33-40}

\bibitem{BCS} D. Bao, S.-S. Chern, Z. Shen. {\it An Introduction to Riemann- Finsler Geometry}. Gradute Text in Mathematics 200, Springer, 2000, xx+ 431p.

\bibitem{BT} Tran Quoc Binh, L. Tamássy, Galloway's compactness theorem on Sasakian manifolds.Aequationes Math. 58 (1999), no. 1-2, 118–124.

\bibitem{Dr1}
S.~Dragomir.
\newblock Submanifolds of {F}insler {S}paces.
\newblock {\em Conf. Sem. Mat. Univ. Bari}, 271:1--15, 1986.

\bibitem{Ga1} G. Galloway, Some results on the occurrence of compact minimal submanifolds, Manuscripta
Math. 35 (1981), 209-219.

\bibitem{Ga2} G.J. Galloway, Compactness criteria for Riemannian manifolds. Proc. of AMS, vol. 84,1,(1982), 6-10.


\bibitem{Ma1}
M.~Matsumoto.
\newblock {\em Foundations of Finsler geometry and special Finsler spaces}.
\newblock Kasheisha Press, Japan, 1986.

\bibitem{Pe1} Ioan Radu Peter,\newblock { On the Morse Index Theorem where the ends are submanifolds in Finsler geometry.}\newblock {\em Houston Journal of Mathematics}, 32(4), 995-1009, 2006.

\bibitem{tip}
F.J. Tipler.\newblock {General relativity and ordinary diferential equations.}\newblock {\em J. Diff. Eq.}, 30, 165-174, 1978.

\bibitem{shen1}
Z. Shen.\newblock {Finsler Geometry of submanifolds}\newblock {\em Math. Ann.},311, 549-576, 1998.

\bibitem{shen2} Z. Shen. {\it Lecture Notes on Finsler Geometry}. Springer Verlag, 2001.
\bibitem{Sz} Z. I. Szab\'o {Positive definite Berwald spaces. }\newblock {\em Tensor (N.S.)}, 35, 25-39, 1981.
\end{thebibliography}
\end{document}